\documentclass{article}

\usepackage{lastpage,graphicx,epstopdf,amsmath,amsfonts,amssymb,setspace,url,epsfig,amsthm,dsfont,upgreek,bbding,makeidx,multicol,multirow,mathrsfs}

\usepackage{amsfonts}
\usepackage{dsfont}
\usepackage{amsmath}
\usepackage{amssymb}
\usepackage{mathrsfs}
\usepackage{graphicx}
\usepackage{amscd}
\newtheorem{thm}{Theorem}

\newtheorem{lem}[thm]{Lemma}

\begin{document}
	
	\title{Generic Existence of Independent Families}
	\author{Michael Perron}
	\maketitle
	
	\begin{abstract}
		We apply the concept of generic existence to p-point, q, and selective independent families that complements and emulates the ultrafilter generic existence results from Canjar and Ketonen.
	\end{abstract}

\section{Introduction}
	All independent families considered in this paper are infinite on a countably infinite set. For convenience we define independent families on $\omega$ and we use notation consistent with \cite{Perron} and \cite{Shelah}. For any set $A \subseteq \omega$ let $A^0 = A$ and $A^1 = \omega \setminus A$. For any set $I$ let FF($I$) = $\{h : h: I \rightarrow 2,$ dom($h$) is finite $\}$. For any set $I$ and any $h \in$ FF($I$) let $I^h = \bigcap \{ A^{h(A)} : A \in$ dom$(I) \}$. An infinite family of sets $I$ is called independent if for every $h \in$ FF($I$) the set $I^h$ is non-empty. If $I$ is an independent family we call ENV($I) = \{I^h : h \in$ FF($I)\}$ the envelope of $I$.
	
	An independent family is called selective if every function on $\omega$ is either one-to-one or constant on a set in the envelope. We say that an independent family is everywhere selective if for every function $f$ on $\omega$ and every $h \in$ FF($I$) we can extend $h$ to some $k \in$ FF($I$) such that $f$ is one-to-one or constant on $I^k$. A p-point-independent family is one where every function on $\omega$ is either finite-to-one or constant on a set in its envelope, and a q-independent family is one where every finite-to-one function on $\omega$ is one-to-one on a set in its envelope. We can define everywhere p-point and everywhere q in the same fashion as we defined everwhere selective.
	
	Independent families have a naturally associated filter. If $I$ is an independent family then Fil$_I = \{ A \subseteq \omega:$ if $h \in$ FF($I$) then $h$ can be extended to $k \in$ FF($I$) such that $I^k \subseteq A \}$ is a filter. We say that an independent family $I$ has a q-filter if every finite-to-one function on $\omega$ is one-to-one on a set in Fil$_I$.
	
	It is consistent that all of the various types of independent families we defined exist under CH, and several diamond principles (see \cite{Perron} and \cite{Moore}). For the consistency of non-existence, Shelah in \cite{Shelah2} gave a model where q-filters do not exist so it is consistent that independent families with q-filters do not exist.
	
	Since we are closely following a paper based on ultrafilters, which are maximal filters, we will mention how our defined independent families related to maximal independent families. A maximal independent family is any independent family $I$ such that if $A \subseteq \omega$ and $A \not \in I$ then $I \cup \{ A\}$ is not independent. Selective independent and p-point-independent families are maximal, but q-independent families need not be (see \cite{Perron}).
	
	Let $\mathfrak{c}$ denote the cardinality of all functions on $\omega$. If $f$ and $g$ are functions on $\omega$ we say that $f$ dominates $g$ if for all $n \in \omega$ $f(n) > g(n)$. We say a family $F$ is a dominating family if for every function $f$ on $\omega$ there is a $g \in F$ such that $g$ dominates $f$. Let $\mathfrak{d}$ be the minimum cardinality of a dominating family of functions on $\omega$. Let $\mathfrak{m}_c$ denote the minimum cardinality of a family $H$ of functions on $\omega$ where for any function $g$ on $\omega$ there is some $h \in H$ such that $g(n) \not = h(n)$ for all $n \in \omega$. From the definitions it is apparent that $\mathfrak{m}_c \leq \mathfrak{d} \leq \mathfrak{c}$.
	
    The cardinal $\mathfrak{m}_c$ is equivalent to, and get its namesake from, Martin's axiom for countable partial orders, where if $P$ is a countable partial order and $\mathcal{D}$ is a collection of less than $\mathfrak{m}_c$-many dense sets on $P$ then there is a filter $G$ on $P$ where $G \cap D \not = \emptyset$ for every $D \in \mathcal{D}$.

	%A family of sets $I$ is called independent if whenever $X$ and $Y$ are finite subsets of $I$ then $\bigcap X \setminus \bigcup Y \not = \emptyset$.
	
	%Why is it important to study independent families?
	
	%Currently very little has been known. The results from this paper strongly suggest a rich field to explore. This field can be guided by the exploration of ultrafilters.
	%Usefulness
	%Similar but different to objects currently being studied

\section{Results}

For an uncountable cardinal $\kappa$ we say that selective independent families exist $\kappa$-generically if whenever $I$ is an independent family with $|I| < \kappa$ then $I$ can be extended to a selective independent family. When $\kappa = \mathfrak{c}$ then we just say that selective independent families exist generically. The generic existence of other varieties of independent families are defined similarly.

\begin{thm}
	The following are equivalent:
    \begin{enumerate}
    	\item $m_c = \mathfrak{c}$
    	\item Everywhere selective independent families exist generically.
    	\item Selective independent families exist generically.
    \end{enumerate}
\end{thm}

\begin{thm}
	The following are equivalent:
	\begin{enumerate}
		\item $\mathfrak{d} = \mathfrak{c}$
		\item Everywhere p-point-independent families exist generically.
		\item P-point-independent families exist generically.
	\end{enumerate}
\end{thm}

\begin{thm}
	The following are equivalent:
	\begin{enumerate}
		\item $\mathfrak{m}_c = \mathfrak{d}$
		\item Independent families with q-filters exist $\mathfrak{d}$-generically.
		\item Everywhere q-independent families exist $\mathfrak{d}$-generically.
		\item Q-independent families exist $\mathfrak{d}$-generically.
	\end{enumerate}
\end{thm}

These statements improve previous results from \cite{Perron}.

\section{Proofs}

We present the proofs to the theorems as a series of lemmas, and we need the following definition. If $f$ and $g$ are functions of $\omega$ we say that $f$ and $g$ are almost disjoint if $\{ n \in \omega : f(n) = g(n) \}$ is finite. A collection $A$ of functions on $\omega$ is called almost disjoint if each pair of elements from $A$ is almost disjoint.

For Theorem 1, 1 $\rightarrow$ 2 has been demonstrated as Cor 59  in \cite {Perron}, and it is clear by the definitions that 2 $\rightarrow$ 3. To complete the theorem we follow the work of Canjar in \cite{Canjar}.

\begin{lem}[3 $\rightarrow 1$ in Theorem 1]
	If selective independent families exist generically then $m_c = \mathfrak{c}$.
\end{lem}

\begin{proof}
	Let $H$ be a collection of less than $\mathfrak{c}$-many functions on $\omega$ and we look for a function $g$ on $\omega$ such that for every $h \in H$ $g(b) = h(b)$ for some $b \in \omega$. We can assume without loss of generalization that $H$ is a collection of almost-disjoint functions.
	
	Let $\{B_n : n \in \omega \}$ be a partition of $\omega$ such that $|B_n| = n$. Let $X_n = \{p : p: B_n \rightarrow \omega \}$ and $X = \bigcup \{X_n : n \in \omega \}$. Note that $X$ is a countable set, and we will work with independent families on $X$. Define a function $N : X \rightarrow \omega$ via $N(p) = n$ where $n$ is the unique $n \in \omega$ such that $p \in X_n$.

	Let $J = \{C_k : k \in \omega \}$ be a countable independent family on $\omega$ and we create an independent family $J^\prime$ on $X$. Define a map $\phi$ where if $A \subseteq \omega$ then $\phi(A) = \bigcup \{X_n : n \in A)$. Note that the image of a non-empty set is non-empty, and the image of an intersection of sets will the intersection of the images. Thus $J^\prime = \{ \phi(C_k) : k \in \omega \}$ is an independent family.

	For $h \in H$ define $A_h = \{ p \in X: p(b)=h(b)$ for some $b \in B_{N(p)} \}$. Let $I = \{A_h : h \in H \} \cup J^\prime$ and we show that $I$ is an independent family on $X$.
	
	Let $E \in ENV(J^\prime)$ and $X, Y \subseteq H$ such that $X$ and $Y$ are finite, and $X \cap Y = \emptyset$. Find $n \in \omega$ large enough such that $X_n \subseteq E, |X| < n$, and $x(k) \not = y(k)$ for every $x \in X, y \in Y$, and $k \in B_n$. We can then find $p \in X_n$ such that $p \in A_x$ and $p \not \in A_y$ for every $x \in X$ and $y \in Y$.
	
	By our assumption, extend $I$ to a selective independent family $I^\prime$. Find $E \in$ ENV($I^\prime$) such that $N \upharpoonright E$ is 1-1 or constant. This cannot be constant because we included $J^\prime$. The partial function $f = \bigcup \{p : p \in E \}$ then agrees with all but finitely many $h \in H$, and since the domain of $f$ is co-countable we can extend $f$ to a function $g$ that agrees with all $h \in H$.

	% Proof for the WLOG:
	% Let $X = \{^n \omega : n \in \omega \}.$ For $h \in H$, let $h_n : n \rightarrow \omega$ via $h(k) = h_n(k)$ for $1 \leq k \leq n$ and define $h^\prime = \{ h_n : n \in \omega \}$. Then $H^\prime = \{ h^\prime : h \in H \}$ is almost-disjoint.
	% Find $g^\prime : \omega \rightarrow X$ and define $g(n) = g^\prime (n)(n)$. If $h \in H$ then there is an $n \in \omega$ such that $g^\prime (n) = h^\prime (n) = h_n$. Then $g(n) = g^\prime(n)(n) = h^\prime(n)(n) = h_n(n) = h(n)$.
\end{proof}

For Theorem 2, only $1 \rightarrow 3$ was explicitly demonstrated as Theorem 62 in \cite{Perron}, and we leave the proof that 1 $\rightarrow$ 2 as an exercise to the reader.

\begin{lem}[3 $\rightarrow 1$ in Theorem 2]
	If p-independent families exist generically then $\mathfrak{d} = \mathfrak{c}$.
\end{lem}

\begin{proof}
	As before, let $H$ be a collection of less than $\mathfrak{c}$-many functions and assume that $H$ is almost disjoint. We show that $H$ cannot be a dominating family.
	
	Find $B_n, X_n, X, N, J$, and $I$ as before.
	
	Extend $I$ to a p-independent family, so find $E \in$ ENV($I^\prime$) such that $N \upharpoonright E$ is finite-to-one. For any $n \in \omega$ there is only finitely many $p \in E$ such that $n \in$ dom($p$) so we can find a partial function $f$ such that $f(n) > p(n)$ for any $p \in E$ and any $n \in$ dom($p$). Again, $f$ has co-countable domain and is not dominated by all but finitely many $h \in H$ so we can extend $f$ to a function $g$ that is not dominated by any $h \in H$.
\end{proof}

For Theorem 3, only 1 $\rightarrow$ 4 had been demonstrated so we present a proof of $1 \rightarrow 2$. To do so we need a lemma for how to extend an independent family in the face of a finite-to-one function.

\begin{lem}
	If $I$ is an independent family with $|I| < \mathfrak{m}_c$ and $f$ is a finite to one function on $\omega$ then there is an independent family $J$ and a set $A \subseteq \omega$ such that $I \subseteq J$, $A \in$ Fil$_{J}$, and A is $f$-rare.
\end{lem}

\begin{proof}
	Let $P = \{F \subseteq \omega : F$ is finite and $F$ is $f$-rare $\}$. For each $E \in$ ENV($I$) the set $D_E = \{F \in P: F \cap E \not = \emptyset\}$ is dense in $P$. Since $|$ENV($I$)$| = |I| < \mathfrak{m}_c$ then we can find a generic filter $G$ on $P$ such that $G \cap D_E$ for each $E \in$ ENV($I$). Let $A = \bigcup G$. Note that $A$ is $f$-rare.

	Since $G \cap D_E$ for each $E \in$ ENV($I$) then $A \cap E \not = \emptyset$. Find a countable independent family $\{ J_n : n \in \omega \}$ such that $J_n \subseteq A$ and $I \cup \{ J_n : n \in \omega \}$ is independent. Then $J = I \cup \{J_n : n \in \omega \}$ works as needed.
\end{proof}

\begin{lem}[1 $\rightarrow 2$ in Theorem 3]
	If $m_c = \mathfrak{d}$ then independent families with q-filters exist $\mathfrak{d}$-generically.
\end{lem}

\begin{proof}
	Let $D = \{f_\alpha : \alpha < \mathfrak{d} \}$ be a dominating family and we work by induction. Let $I_0$ be an independent family with $|I_0| < \mathfrak{d}$. When $\alpha$ is a limit define $I_\alpha = \bigcup_{\beta < \alpha} I_\beta$. For any $\alpha < \mathfrak{d}$ we can, by the previous lemma, find $A_\alpha$ and extend $I_\alpha$ to $I_{\alpha + 1}$ such that $A_\alpha$ is $f_\alpha$-rare. Then $I_\mathfrak{d}$ has a q-filter.
	
	To see, let $f$ be a finite-to-one function on $\omega$. Define $\bar{f}(x) = max\{n \in \omega : f(n) = f(x)\}$. Then there is a some $f_\alpha$ that dominates $\bar{f}$, and there a set $A_\alpha \in$ Fil$_{I_\alpha}$ such that $E_\alpha$ is $f_\alpha$ rare. Since $f_\alpha$ dominates $\bar{f}$ then $A_\alpha$ is also $\bar{f}$-rare. Then $f \restriction A_\alpha$ must be one to one since if $a, b \in A_\alpha$ with $a < b$ then $\bar{f}(a) < b$ so $f(a) \not = f(b)$.
\end{proof}

It is apparent by the definitions that $2 \rightarrow 3$ and $3 \rightarrow 4$.

\begin{lem}[4 $\rightarrow 1$ in Theorem 3]
	If q-independent families exist $\mathfrak{d}$-generically then $m_c = \mathfrak{d}$.
\end{lem}

\begin{proof}
	As before, let $H$ be a collection of less than $\mathfrak{d}$-many functions and assume that $H$ is almost disjoint. Find $\{B_n : n \in \omega \}$ as before. For every finite subset $S$ of $H$, define $f_S : \omega \rightarrow \omega$ via $f_S(k) = max \{f(i): f \in S$ and $i \in B_n$ where $k \in B_n \}$. Find a function $f$ that is not dominated by the family $\{ f_S : S$ is a finite subset of $H \}$.
	
	Define $X_n = \{ p: p:B_n \rightarrow f(n) \}$, let $X = \bigcup X_n$, and find $N: X \rightarrow \omega$ as before.
	
	Define the independent family $I$ as before, and extend to a $q$-independent family $I^\prime$. Since $N$ is finite-to-one find $E \in$ ENV($I^\prime$) such that $N \upharpoonright E$ is 1-1, and define $g$ as before.
\end{proof}

\section{Remarks}

The existence of our varieties of independent families is not dependent on the generic existence. Selective and everywhere selective independent families of size $\aleph_1$ exist in the Sacks model where $\mathfrak{c}  > \aleph_1$. Likewise q-independent families of size $\aleph_1$ exist in the Random Real model.

In conjunction with the results from Canjar we see that independent families are directly linked to ultrafilters via generic existence. The conditions that allow the generic existence of ultrafilters are exactly those that allow the generic existence of independent families, and vice versa. This suggests a need to continue studying the link between these two types of objects.

\end{document}